\documentclass[10pt, equation]{article}
\usepackage{amsmath,amssymb,amsthm}
\usepackage{mathtools}
\newtheorem{theo}{Theorem}[section]
\newtheorem{lem}{Lemma}[section]

\newtheorem{cor}{Corollary}[section]
\theoremstyle{definition}

\newcommand{\PP}{\Bbb{P}}
\newcommand{\N}{I\!\!N}
\newcommand{\R}{\Bbb{R}}

\newcommand{\E}{\Bbb{E}}

\newcommand{\be}{\begin{equation}}
\newcommand{\ee}{\end{equation}}
\newcommand{\fc}{\mathcal{F}}
\numberwithin{equation}{section}

\begin{document}
\date{}
\title{Feller's upper-lower class test in Euclidean space}
\author{Uwe Einmahl\\
Department of Mathematics, Vrije Universiteit Brussel\\
Pleinlaan 2, B-1050 Brussels, Belgium}
\maketitle

\begin{abstract}
We provide an extension of Feller's  upper-lower class test for  the Hartman-Wintner LIL to the LIL in Euclidean space. We obtain this result as a corollary to a general upper-lower class test  for $\Gamma_n T_n$  where $T_n=\sum_{j=1}^n Z_j$ is a sum of i.i.d. d-dimensional standard normal random vectors and $\Gamma_n$ is a convergent sequence of symmetric non-negative definite $(d,d)$-matrices. In the process we derive  new bounds for the tail probabilities of $d$-dimensional normally distributed random vectors.  \end{abstract}

\noindent \textit{AMS 2020 Subject Classifications:} 60F15.

\noindent \textit{Keywords:}  Hartman-Wintner LIL, Euclidean space, upper-lower class test, Gaussian random vectors.

\section{Introduction} Let $X_1, X_2, \ldots$ be independent identically distributed (= i.i.d.) $d$-dimensional random vectors and set  $S_n=\sum_{j=1}^n X_j, n \ge1.$  Denote the Euclidean norm by $|\cdot|$ and set  $Lt =\log_e(t \vee e), LLt = L(Lt),  t \ge 0.$  Let $\mathcal{N}(\mu,\Sigma)$  be  the $d$-dimensional normal distribution with mean $\mu$ and non-negative definite covariance matrix $\Sigma.$ \\[.2cm]
Assuming  that $\E |X_1|^2 < \infty$ and $\E X_1=0,$  it follows from the  $d$-dimensional version of the  Hartman-Wintner law of the iterated logarithm (LIL)  that with probability one,
\be \label{HWLIL}
\limsup_{n \to \infty} |S_n|/\sqrt{2nLLn} = \sigma,
\ee
where $\sigma^2$ is the largest eigenvalue of $\mathrm{Cov}(X_1)$ (=  covariance matrix of $X_1$). \\[.1cm]
A substantial refinement  of this result is given   by an upper-lower class test.  We say that  a non-decreasing  sequence $\{a_n\}$ of positive real numbers belongs to the upper class (lower class) for $|S_n|$ if we have with probability one,
$|S_n| \le a_n$ eventually ($|S_n| > a_n$ infinitely often). By the 0-1 law of Kolmogorov any such sequence $\{a_n\}$ has to be in one of these two classes. \\[.1cm]
The above LIL result then only implies that 
$a\sqrt{2nLLn}$ belongs to the upper (lower) class if $a > \sigma$  ($a < \sigma$), but, for instance, it is not clear from the LIL to what class the sequence $\sigma \sqrt{2nLLn}$ belongs. 
If one has normally distributed random vectors this sequence belongs to the lower class for $|S_n|$. This easily follows from the subsequent upper-lower class test which is part of a  more general result  in \cite{Kue}.\\[.2cm]
{\bf Theorem A (\cite{Kue}}) {\it If $X_n, n \ge 1$ is a sequence of independent   $\mathcal{N}(0,\Sigma)$-distributed random vectors we have for any non-decreasing sequence $\phi_n$ of positive real numbers,  
$$\PP\{|S_n| \le \sigma \sqrt{n}\phi_n \mbox{ eventually}\} =1 \mbox{ or } 0$$
according as 
$$ \sum_{n=1}^{\infty}n^{-1} \phi_n^{d_1}\exp(-\phi_n^2/2) < \infty \mbox{ or} =\infty,$$
where $d_1$ is the multiplicity of the largest eigenvalue $\sigma^2$ of $\Sigma$.}\\[.1cm]
 It is natural now to ask whether this result also holds for arbitrary mean zero random vectors satisfying the minimal condition $\E|X_1|^2 < \infty.$\\[.2cm]
In the $1$-dimensional case this problem was  investigated by Feller \cite{Fe} in 1946 who showed that the classical Kolmogorov-Erd\H os-Petrowski integral test (= the 1-dimensional version of the above upper-lower class test) remains valid if
\be \label{Fe1}
\E X_1^2 I\{|X_1| \ge t\} = O((LLt)^{-1}) \mbox{ as }t \to \infty. 
\ee
Moreover, he proved that this upper-lower class test holds in general if one replaces $\sqrt{n}$ by  $\sqrt{B_n}$ where $B_n = \sum_{j=1}^n \sigma_j^2$ and $\sigma_j^2 = \E X_1^2 I\{|X_1| \le \sqrt{j}/(LLj)^4\}, j \ge 1$, that is, we have 
$$\PP\{|S_n| \le  \sqrt{B_n}\phi_n \mbox{ eventually}\} =1 \mbox{ or } 0$$
according as 
$$ \sum_{n=1}^{\infty}n^{-1} \phi_n \exp(-\phi_n^2/2) < \infty \mbox{ or} =\infty.$$
Feller's proof was based on a  skilful truncation argument which unfortunately only worked for random variables having distributions which are symmetric about zero. It was an  open problem for quite some time whether Feller's results also hold in the non-symmetric case. This question was  settled in \cite{E-5} where an extension of Feller's truncation arguments to the non-symmetric case was found.\\
Feller  also stated an interesting conjecture in his paper \cite{Fe}, namely that  the above upper-lower class test still holds if one replaces $B_n$ by $n\sigma_n^2$.\\  This conjecture has finally been proven in  \cite{Di-E}  leading to the following  result in dimension 1,
\be \label{Fe-final}
\PP\{|S_n| \le \sigma_n \sqrt{n}\phi_n \mbox{ eventually}\} =1 \mbox{ or } 0\ee
according as 
$$\sum_{n=1}^{\infty}n^{-1} \phi_n\exp(-\phi_n^2/2) < \infty \mbox{ or} =\infty.$$
Moreover, it was shown in \cite{Di-E} that  (\ref{Fe-final}) holds if 
$\sigma_n^2 = \E X_1^2 I\{|X_1| \le c_n\}$
where $c_n$ can be any  non-decreasing sequence of positive real numbers satisfying for a suitable sequence $\epsilon_n \to 0$ and large $n,$
\be \label{c_n}
\exp(-(Ln)^{\epsilon_n}) \le c_n/\sqrt{n} \le \exp((Ln)^{\epsilon_n})
\ee
From this result (and also from Feller's  original result) it easily follows that the Kolmogorov-Erd\H os-Petrowski upper-lower class test is applicable for sums $S_n$  of i.i.d. random variables  which satisfy condition (\ref{Fe1}). If this condition is not satisfied the sums $S_n$ are in a certain sense  {\it smaller} than sums of i.i.d normal random variables. There are examples where, for instance, the sequence $\sigma \sqrt{2nLLn}$
belongs to the upper class for $|S_n|.$\\[.2cm]
We return to the $d$-dimensional case:\\[.1cm] Feller's upper-lower class test (\ref{Fe-final}) was actually obtained in \cite{Di-E}  as a corollary to a $d$-dimensional result. To formulate this result we need some extra notation:  we are now assuming that $X_n=(X_n^{(1)},\ldots, X_n^{(d)})$ are i.i.d.   mean zero random vectors such that $\E|X_1|^2 < \infty$ and  we define a sequence $\Gamma_n$  of symmetric non-negative definite matrices as follows,
\be \label{gamma_n}
\Gamma_n^2 =\left[\E X_1^{(i)}X_1^{(j)} I\{|X_1| \le c_n\}\right]_{1 \le i,j \le d},
\ee
where again $c_n, n \ge 1$ is a non-decreasing sequence of positive real numbers satisfying condition (\ref{c_n}) for some sequence $\epsilon_n \to 0.$\\
Then we clearly have $\Gamma_n^2 \to \mathrm{Cov}(X_1)$ and assuming that this latter matrix is positive definite we can assume w.l.o.g. that this is also the case for the matrices $\Gamma_n$. \\[.2cm]
{\bf Theorem B (\cite{Di-E})} 
{\it  Let $X_n, n \ge 1$ be i.i.d. mean zero random vectors in $\R^d$ with $\E |X_1|^2 < \infty$ and $\mathrm{Cov}(X_1) = \Gamma^2,$ where $\Gamma$ is a symmetric positive definite (d,d)-matrix. 
 Let  $c_n$ be a non-decreasing sequence of positive real-numbers satisfying condition (\ref{c_n}) for large $n$ and let $\Gamma_n$ be defined as in (\ref{gamma_n}).
 Then we have for any non-decreasing sequence $\phi_n$ of positive numbers,
 $$\PP\{|\Gamma_n^{-1}S_n| \le \sqrt{n}\phi_n \mbox{ eventually}\} = 1 \mbox{ or }=0$$
 according as }
 $$ \sum_{n=1}^{\infty}n^{-1} \phi^d_n\exp(-\phi_n^2/2) < \infty \mbox{ or} =\infty.$$
 \\[.2cm]
 Letting $\lambda_n (\Lambda_n)$ be the smallest (largest) eigenvalue of $\Gamma_n, n \ge 1$ we can infer from Theorem B,
 \be \label{eq21}
 \sum_{n=1}^{\infty}n^{-1} \phi^d_n\exp(-\phi_n^2/2) < \infty \Rightarrow \PP\{|S_n| \le \Lambda_n \sqrt{n}\phi_n \mbox{ eventually}\} = 1
\ee
 and
 \be \label{eq22}
 \sum_{n=1}^{\infty}n^{-1} \phi^d_n\exp(-\phi_n^2/2) = \infty \Rightarrow \PP\{|S_n| >\lambda_n \sqrt{n} \phi_n \mbox{ infinitely often}\} = 1,
\ee
which clearly implies the upper-lower class test (\ref{Fe-final})  in the 1-dimensional case.\\[.2cm]   
In the higher-dimensional case this result provides however only a partial answer. Though the lower class part of  this result holds for any covariance matrix, it is only interesting if the covariance matrix $\Sigma=\Gamma^2$ of $X_1$ is a multiple of the identity matrix. Otherwise, $\lambda_n$ will converge to the smallest eigenvalue of the matrix $\Gamma$ which is strictly smaller then the largest eigenvalue of $\Gamma$. In this case (\ref{eq22})   gives us less information than (\ref{HWLIL}).\\[.2cm]
This has been remedied in \cite{Di} where the following result was proven:\\ Let $\Lambda_n=\lambda_{n,1} \ge \ldots \ge \lambda_{n,d} = \lambda_n$ be the
the eigenvalues of $\Gamma_n$ arranged in a non-decreasing order and taking into account the multiplicities. If $d_1$ denotes the multiplicity of the largest eigenvalue of $\Gamma,$
we have  
\be \label{eq23}
\sum_{n=1}^{\infty}n^{-1} \phi^{d_1}_n\exp(-\phi_n^2/2)  < \infty \Rightarrow \PP\{|S_n| \le \lambda_{n,1} \sqrt{n}\phi_n \mbox{ eventually}\} = 1
\ee
 and
\be \label{eq24}
\sum_{n=1}^{\infty}n^{-1} \phi^{d_1}_n\exp(-\phi_n^2/2) = \infty \Rightarrow \PP\{|S_n| >\lambda_{n,d_1} \sqrt{n} \phi_n \mbox{ infinitely often}\} = 1.
\ee
If $d_1=1$ this  gives  a complete upper-lower class test for $|S_n|.$ Moreover, it implies  the $d$-dimensional version of Feller's result  under condition (\ref{Fe1}) which is due to \cite{E-6}, namely that  the upper-lower class test of Theorem A applies if
\be
\E \langle v, X_1\rangle^2 I\{|\langle v, X_1\rangle| >t \} = O((LLt)^{-1}) \mbox{ as }t \to \infty
\ee
for all vectors $v$ in the eigenspace  of $\Sigma$ determined by its largest eigenvalue $\sigma^2.$ \\[.2cm]
The purpose of the present paper is to completely solve the problem by  establishing a general  upper-lower class test  for the $d$-dimensional version of the Hartman-Wintner LIL.\\[.1cm]
As in \cite{Di} 
a crucial tool for proving this upper-lower class test is  the following  strong invariance principle (see Theorem 2.3 in \cite{Di-E}).\\[.2cm]
{\bf Theorem C}  \emph{Let $\{X_n: n \ge 1\}$  and $\Gamma_n, n \ge 1$ be as in Theorem B. If the underlying p-space $(\Omega,\fc,\PP)$ is rich enough, one can construct independent $\mathcal{N}(0,I)$-distributed random vectors $Z_n, n \ge 1$ such that we have for the sums} $T_n=\sum_{j=1}^n Z_j,$
\begin{itemize}
\item [(a)] $S_n - \Gamma\; T_n = o(\sqrt{n LLn}) \mbox{ as }n \to \infty \mbox{ with prob. 1}$
\item[(b)] $\PP\{|S_n -\Gamma_n\;T_n| \ge 2\sqrt{n}/LLn, |S_n| \ge \frac{4}{3} \|\Gamma\| \sqrt{nLLn} \mbox{ infinitely often}\}=0$
\item[(c)]  $\PP\{|S_n -\Gamma_n\;T_n| \ge 2\sqrt{n}/LLn, |\Gamma\,T_n| \ge \frac{4}{3} \|\Gamma\| \sqrt{nLLn} \mbox{ infinitely often}\}=0$
\end{itemize}
From Theorem C it follows that
 the lower (upper) class of $|S_n|$ matches the lower (upper) class of $|\Gamma_n T_n|$. \\By definition of $\Gamma_n$ we have  $\Gamma^2_n \to \mathrm{Cov}(X_1)=\Gamma^2$ where $\Gamma$ is symmetric and non-negative definite. This implies via Theorem X.1.1 in \cite{Bha}  that $\Gamma_n \to \Gamma$. \\
Further note that we have with probability one,
$$
\limsup_{n \to \infty} |\Gamma_n T_n|/\sqrt{2nLLn} =\sigma.
$$
 whenever $\Gamma_n$ converges to $\Gamma$ with largest eigenvalue $\sigma.$ \\So it is natural to ask whether one can prove an upper-lower class test for $|\Gamma_n T_n|$ where $\Gamma_n$ is a  general convergent sequence and not only 
one of the particular sequences determined by (\ref{gamma_n}). \\This is indeed the case, but we have to impose an extra condition on the sequence  $\Gamma_n$ to make sure that it does  not fluctuate too much. \\[.1cm]
 We 
set for any $\alpha >0$ and $k \ge 1,$
\be \label{DELTA}
\Delta_k(\alpha):=\max_{n_k(\alpha) \le m \le n \le n_{k+1}(\alpha)}\| \Gamma^2_n -\Gamma^2_m\|,
\ee
where $n_k(\alpha)=[\exp(\alpha k/Lk)], k \ge 1$, $[x]$ is the integer part of $x \in \R$ and $\|\cdot\|$ denotes the Euclidean matrix norm, that is $\|A\| = \sup_{|v| \le 1} |Av|$.
\begin{theo} \label{main}
Assume that $d \ge 2.$ Let $\Gamma_n$ be a sequence of symmetric non-negative definite $(d,d)$-matrices  which converges to a symmetric non-negative definite  $(d,d)$-matrix $\Gamma \ne 0$. Assume that for any $\alpha >0,$
 \be \label{main112}
 \sum_{k=1}^{\infty} k^{-\delta}\Delta_k(\alpha) < \infty, \;\forall \delta >0.
 \ee
 If $T_n=\sum_{j=1}^n Z_j, n \ge 1$ where $\{Z_n\}$ is a sequence of independent $\mathcal{N}(0,I)$-distributed random vectors,  we have for any non-decreasing  sequence $\phi_n$ of positive real numbers,
$$\PP\{|\Gamma_n T_n| \le  \sqrt{n} \phi_n \mbox{ eventually}\} = 1 \mbox{ or } =0$$
according as
$$\sum_{n=1}^{\infty} \frac{\phi_n}{n \lambda_{n,1}}\prod_{i=2}^{d _1} \left(\frac{\lambda_{n,1}}{(\lambda_{n,1}^2-\lambda_{n,i}^2)^{1/2}} \wedge  \frac{\phi_n}{\lambda_{n,1}}\right) \exp\left(-\frac{\phi_n^2}{2\lambda_{n,1}^2} \right) < \infty \mbox{ or }=\infty,$$
where  $\lambda_{n,1} \ge \ldots \ge \lambda_{n,d}$ are the eigenvalues of $\Gamma_n$ and $d_1$  is the multiplicity of the largest eigenvalue of the matrix $\Gamma$.\end{theo}
\noindent
As usual,  we set $a/0=\infty, a >0$ and $\prod_{i=2}^1 a_i=1.$\\[.2cm]
Combining Theorem \ref{main} and Theorem C, we obtain the following multi-dimensional version of Feller's upper-lower class test (\ref{Fe-final}).
\begin{cor} \label{CFeller} 
Assume that $d \ge 2.$ Let $X_n, n \ge 1$ be a sequence of i.i.d. random vectors with $\E|X_1|^2< \infty,$ $\E X_1=0$ and  non-negative definite covariance matrix $\Gamma^2 \ne 0.$ Then  we have
$$\PP\{|S_n| \le  \lambda_{n,1} \sqrt{n} \phi_n \mbox{ eventually}\} = 1 \mbox{ or } =0$$
according as
$$\sum_{n=1}^{\infty} \frac{\phi_n}{n}\prod_{i=2}^{d _1} \left(\frac{\lambda_{n,1}}{(\lambda_{n,1}^2-\lambda_{n,i}^2)^{1/2}} \wedge  \phi_n\right) \exp\left(-\frac{\phi_n^2}{2} \right) < \infty \mbox{ or }=\infty,$$
where  $\lambda_{n,1} \ge \ldots \ge \lambda_{n,d}$ are the eigenvalues of the matrices  $\Gamma_n, n \ge 1$ defined as in (\ref{gamma_n}) and  $d_1$ is the multiplicity of the largest eigenvalue of $\Gamma^2$.
 \end{cor}
\noindent Corollary \ref{CFeller} immediately implies an improved version of (\ref{eq23}). Likewise, one can infer (\ref{eq24}) from Theorem \ref{main} (see section 3.3).\\[.2cm]
Though Corollary \ref{CFeller} provides a final result concerning the $d$-dimensional version of the Hartman-Winter LIL, several related problems remain open. A natural question is whether one can obtain a similar integral test refinement for the LIL in infinite-dimensional Hilbert spaces. We conjecture that this is possible under a second moment assumption (see \cite{E-6}). But in  the infinite-dimensional setting  a finite second moment is not necessary for the LIL.  No complete upper-lower class test seems to be known for the LIL in Hilbert space if the  second moment is infinite (see \cite{Ein1991, Ein1992}). In \cite{E-Li} LIL type results were obtained with different normalizers such as $\sqrt{n}(LLn)^p, p > 1/2$ and the next task would be to determine the upper and lower classes in this case as well. Finally, there is also an upper-lower class test for weighted sums of independent random variables (see \cite{bhw}) under an extra assumption which is slightly more restrictive than (\ref{Fe1}). It is not clear whether a  version of this result exists under the minimal assumption of a finite second moment and what such a result would look like. \\[.2cm]
The remaining part of the paper is organised as follows: In Section 2 we formulate and prove  the new results for Gaussian random vectors. The main result is Theorem \ref{LemTailBoundsDistriGammanZn} which should be of independent interest. In Section 3 we prove all other results:  The proof of Theorem \ref{main} can be found  in subsection 3.1 (upper class part) and  in subsection  3.2 (lower class part). Corollary \ref{CFeller} and (\ref{eq24}) are proven in subsection 3.3. Finally Lemma \ref{LemSubseqPhi_nEnoughUpper_or_LowerCl} which is crucial for the proof of Theorem \ref{main} is proven in subsection 3.4.
\section{Some results on Gaussian random vectors}
\subsection{Bounds related to Zolotarev's result} \label{fixed cov}
  Let $Z: \Omega \to \R^d$ be a $d$-dimensional $\mathcal{N}(0, I )$-distributed random vector. Then  $|Z|^2$ has a 
	a chi-square distribution with $d$ degrees of freedom and we have an explicit formula for the 
density  $f_d$ of $|Z|^2$, namely, 
\be \label{EqDensChi^2}
		f_{d}(z) = C_0 z^{d/2-1}\exp(-z/2), z > 0,
	\ee
where $C_0=C_0(d)= 2^{-d/2} \Gamma(d/2)^{-1}.$	\\This implies via  partial integration that there exist  positive constants $C_1$ and $C_2$ depending on $d$ only such that 
		 \be \label{eq42}
		C_1 t^{d-2}\exp(-t^2/2) 
	\le 
		\PP\{|Z| \ge t\} 
	\le 
		C_2 t^{d-2}\exp(-t^2/2), t \ge 2d.
	\ee
	Consider now more generally a  random vector $Y: \Omega \to \R^d$ having  an  $\mathcal{N}(0, \Gamma^2)$-distribution where $\Gamma\ne 0$ is a general non-negative definite  matrix. 
	Let $\{e_1,\ldots,e_d\}$ be an orthonormal basis of $\R^d$ consisting of eigenvectors  corresponding to  the eigenvalues $\lambda_1  \ge \lambda_2 \ge \ldots \ge \lambda_d$ of $\Gamma$.  Then we clearly have
\be \label{fourier1}
Y=\sum_{i=1}^d  \langle Y, e_i\rangle e_i:=  \sum_{i=1}^d \lambda_i \eta_i e_i,
\ee
where $\eta_i, 1 \le i \le d$ are independent 1-dimensional standard normal random variables.
Moreover it follows that
\be \label{fourier2}
 |Y|^2 =\sum_{i=1}^d \lambda_i^2 \eta_i^2
\ee
 Due to  a classical result of Zolotarev \cite{zo}  we have for the density $h$  of $|Y|^2$
	\be \label{zol}
	h(z) \sim K(\Gamma^2)\lambda_1^{-2} f_{d_1}(z/\lambda_1^2) \mbox{ as } z \to \infty,\ee
	where 
	$$K(\Gamma^2)= \prod_{i=d_1 +1}^d (1- \lambda_i^2/\lambda_1^2)^{-1/2}$$ and $d_1$ is the multiplicity of the largest eigenvalue $\lambda_1$ of $\Gamma$. If $\Gamma^2$ is a diagonal matrix, we set $K(\Gamma^2)=1$. \\[.1cm]
In view of (\ref{zol}) it is clear that the tail probabilities $\PP\{|Y| \ge t\}$ are of the same order as $\PP\{|Z'| \ge t/\lambda_1\}$ for large $t$, where $Z'$ is a $d_1$-dimensional $\mathcal{N}(0,I)$-distributed random vector.\\[.2cm] In order to obtain concrete bounds for $\PP\{|Y| \ge t\}$, Einmahl (\cite{Ein1991}) derived 
the following inequalities  for  $h$ (see Lemma 1 and Lemma 3(b) in \cite{Ein1991})
\be \label{eq43}
h(z) \ge \frac{1}{4} K(\Gamma^2) f_{d_1}(z/\lambda_1^2)/\lambda_1^2, z \ge 2d_1 \E[|Y|^2]/(1-\lambda_{d_1 +1}^2/\lambda_1^2)
\ee
and moreover if $d_1 \ge 2,$
\be \label{eq44}
h(z) \le K(\Gamma^2) f_{d_1}(z/\lambda_1^2)/\lambda_1^2, z >0.
\ee
We still need an upper bound for $h$ if $d_1=1$. The bound given in Lemma 3(a) of \cite{Ein1991}  is not sufficient for our present purposes.\\

\begin{lem}\label{LemDensBound-1}
Assume that $d\ge 2$ and let $Y$ be an $\mathcal{N}( 0, \Gamma^2)$-distributed random vector
where the largest eigenvalue $\lambda_1^2$ of $\Gamma^2$ has multiplicity 1. Then we have for the density $h$ of $|Y|^2$,
\be \label{d_1=1}
h(z) \le C_3 K(\Gamma^2) f_1(z/\lambda_1^2)/\lambda_1^2, z > 0,
\ee
where $C_3=C_3(d) > 1$ is a constant depending on $d$ only.
 \end{lem}
\begin{proof} To simplify notation we set  $\rho_i^2 =\lambda_{i}^2/\lambda_1^2 <1, 2 \le i \le d.$\\[.1cm]
(i)  We first look at the case where the eigenvalue $\lambda_2^2$ has multiplicity $d-1$ so that there are only two different eigenvalues. 
In this case it follows from (\ref{fourier2}) that
$$|Y|^2/\lambda_1^2 = \eta^2_1 + \rho_2^2 R_2,$$
where the random variables   $\eta_1^2$ and $R_2=\sum_{i=2}^d \eta_i^2$ are independent and  
	 have chi-square distributions with $1$ and $d_2:=d-1$ degrees of freedom, respectively. \\[.1cm] Consequently, $\eta_1^2$ and $\rho_2^2 R_2$ have densities $f_1$ and $y \mapsto f_{d_2}(y/\rho_2^2)/\rho_2^2$, respectively.\\
Applying  the convolution formula   we see that the density of $|Y|^2/\lambda_1^2 $ is equal to	$$
		g(z,\rho_2,d_2):=
		\rho_2^{-2}\int_0^z f_1(z-y)f_{d_2}(y/\rho_2^2)dy, z \ge 0.
	$$
	We can infer that
	\begin{eqnarray*}
      &&\frac{g(z,\rho_2,d_2)}{f_1(z)} 
		=
		\frac{z^{d_2/2- 1}}{2^{d_2/2}\Gamma(d_2/2)\rho_2^{d_2}} \int_0^z (1-y/z)^{-1/2} (y/z)^{d_2/2 -1} e^{-(\rho_2^{-2} -1)y/2} dy\\
	&\le& 
		\frac{1}{2^{d_2/2}\Gamma(d_2/2)\rho_2^{d_2}}\bigg[\sqrt{2}\int_0^{z/2}\frac{ y^{d_2/2 -1}}{ e^{(\rho_2^{-2} -1)y/2}} dy \\
		&&\hspace{3cm}
		+  
		\frac{ z^{d_2/2}}{e^{(\rho_2^{-2} -1)z/4}}\int_{1/2}^1 t^{d_2/2-1} (1-t)^{-1/2}dt\bigg]\\
	&\le&\frac{1}{2^{d_2/2}\Gamma(d_2/2)\rho_2^{d_2}} \left[\sqrt{2}\int_0^{\infty}\frac{ y^{d_2/2 -1}}{ e^{(\rho_2^{-2} -1)y/2}} dy 
		+  
		\frac{ z^{d_2/2}}{e^{(\rho_2^{-2} -1)z/4}}B(1/2, d_2/2)\right]\\
		&=&(1-\rho_2^2)^{-d_2/2} \left[\sqrt{2}+ \left(z(\rho_2^{-2}-1)/2\right)^{d_2/2}e^{-(\rho_2^{-2} -1)z/4}\frac{\Gamma(1/2)}{\Gamma((d_2+1)/2)}\right]
	\end{eqnarray*}
Recall that we have for the beta function 
$$B(\alpha_1,\alpha_2):=\int_0^1 t^{\alpha_1-1}(1-t)^{\alpha_2-2}dt=\frac{\Gamma(\alpha_1)\Gamma(\alpha_2)}{\Gamma(\alpha_1 + \alpha_2)}, \alpha_1,\alpha_2 >0.$$
Using that $\Gamma(1/2)=\sqrt{\pi}$ and
$t^{d_2/2}e^{-t/2} \le (d_2/e)^{d_2/2}, t > 0$,
we find that 
$$g(z,\rho_2,d_2) \le  C_4(d_2)(1-\rho_2^2)^{-d_2/2} f_1(z),$$
where $C_4(d_2) =\sqrt{2} + \sqrt{\pi}(d_2/e)^{d_2/2} /\Gamma((d_2+1)/2).$ \\[.1cm]
As we have $h(z)= g(z/\lambda_1^2,\rho_2,d_2)/\lambda_1^2, z > 0$ and $K(\Gamma^2) = (1 - \rho_2^2)^{-d_2/2}$ in this case, we obtain the assertion if there are only two different eigenvalues.\\[.2cm]
(ii) Suppose now that there are three different eigenvalues $\lambda_1^2 > \lambda_2^2 > \lambda_d^2$ with respective multiplicities $1, d_2$ and $d_3$ so that $d_2 + d_3 = d-1.$ Then:
$$|Y|^2/\lambda_1^2 = \eta_1^2 + \rho_2^2\sum_{i=2}^{d_2+1}\eta_i^2 +  \rho_d^2\sum_{i=d_2+2}^{d}\eta_i^2 =:(\eta^2_1 + \rho_2^2 R_2) + \rho_d^2 R_3,$$
where we already know by (i)  that the density $h_1$ of $\eta^2_1 + \rho_2^2 R_2$ satisfies
$$h_1(z) \le C_4(d_2)(1 -\rho_2^2)^{-d_2/2} f_1(z), z > 0.$$
Denoting the density of $|Y|^2/\lambda_1^2$ by $h_2$, we have,
\begin{eqnarray*}
h_2(z)/f_1(z)  &=& \rho_d^{-2}\int_0^z h_1(z-y)f_{d_3}(y/\rho_d^2)dy /f_1(z)\\
&\le&C_4(d_2)(1 -\rho_2^2)^{-d_2/2}g(z,\rho_d,d_3)\\
&\le& C_4(d_2)C_4(d_3)(1 -\rho_2^2)^{-d_2/2}(1 -\rho_d^2)^{-d_3/2}.
\end{eqnarray*}
It follows that 
$$h(z)=h_2(z/\lambda_1^2)/\lambda_1^2  \le C_4(d_2)C_4(d_3) K(\Gamma^2) f_1(z/\lambda_1^2)/\lambda_1^2, z > 0.$$
(iii) If $\Gamma^2$ has $ r $ different eigenvalues $\lambda^2_i$ with multiplicities $d_1=1$ and $d_i, 2 \le i \le r$ we get via induction:
$$h(z) \le \prod_{i=2}^r C_4(d_i) K(\Gamma^2) f_1(z/\lambda_1^2)/\lambda_1^2, z > 0.$$ Setting $C_3(d) =\max\{ \prod_{i=2}^r C_4(d_i): d_2 + \ldots + d_r = d-1\}$  the assertion of the lemma follows. 
\end{proof}
\noindent Integrating the inequalities (\ref{eq43}), (\ref{eq44}) and (\ref{d_1=1}), one obtains upper and lower bounds for the probabilities $\PP\{|Y| \ge t\}$, but  the lower bound is only valid if  $$|t| \ge \sqrt{2d_1} \E[|Y|^2]^{1/2}/(1-\lambda_{d_1 +1}^2/\lambda_1^2)^{1/2}.$$
This can be a problem if  one looks at a sequence $Y_n: \Omega \to \R^d$ of random vectors such that $Y_n$ is $\mathcal{N}(0,\Gamma_n^2)$-distributed and where
$
\Gamma_n \to \Gamma$
for    non-negative definite symmetric matrices $\Gamma_n, n \ge 1$ and  $\Gamma$. \\
Let $\lambda_{n,1} \ge \ldots \ge \lambda_{n,d}$ be the eigenvalues of $\Gamma_n, n \ge 1$.
Then we have as $n \to \infty,$
\be \label{ga2_n}
\lambda_{n,i} \to \lambda_i, 1 \le i \le d,
\ee
where $\lambda_1\ge \ldots \ge \lambda_d$ are the eigenvalues of $\Gamma$. (For a proof see, for instance, Lemma A.1 in \cite{Ein1991}.) \\
 The multiplicity $d_{n,1}$ of the largest eigenvalue of $\Gamma_n$ can be different from the multiplicity $d_1$ of the largest eigenvalue $\lambda_1$ of $\Gamma$ and we do not have necessarily that $d_{n,1} \to d_1$.
 Moreover, it is possible that   $(1 - \lambda_{n,d_{n,1}+1}^2/\lambda_{n,1}^2) \to 0$ so that we do not always have a lower bound for $\PP\{|Y_n| \ge t\}$ if $t$ is  small.\\
 
\subsection{Better bounds}
To overcome this difficulty we define for any given $t \ge 3d \lambda_1$ a regularized version $Y_t: \Omega \to \R^d$ of the random vector $Y: \Omega \to \R^d$ where we merge all eigenspaces  determined by  the eigenvalues of $\Gamma$ which are sufficiently close to the largest eigenvalue of $\Gamma$.  \\
Let $\{e_1,\ldots,e_d\} \subset \R^d$  and $\eta_1, \ldots, \eta_d: \Omega \to \R$ be defined as in (\ref{fourier1}).\\
We set
\be \label{Y_t}
Y_{t} = \sum_{i=1}^{\tilde{d}} \lambda_{1}\eta_{i} e_{i} + \sum_{i=\tilde{d} +1}^{d} \lambda_{i} \eta_{i} e_{i},\ee
where
$$\tilde{d}=\tilde{d}(t)= \max\{1 \le i \le d: \lambda_{1}^2 - \lambda_{i}^2 \le 4d^2 \lambda_{1}^4 t^{-2}\}.$$
Since we are assuming that $t \ge 3d\lambda_{1},$ we  have in particular
\be \label{(3.8)}
\lambda^2_{\tilde{d}} \ge \lambda^2_{1}/2.
\ee
It is easy to see from the above definition that  $Y_t$ is normally distributed with mean zero. 
Writing the covariance matrix of $Y_t$ as $\Gamma_t^2$, where $\Gamma_t$ is symmetric and non-negative definite,  the eigenvalues  $\lambda_{t,1}\ge \ldots \ge \lambda_{t,d}$ of the  matrix $\Gamma_{t}$  are related to the eigenvalues of $\Gamma$ as follows,
$$\lambda_{t,i}=\lambda_{1}, 1 \le i \le \tilde{d}$$
and if $\tilde{d} < d$, we further have
$$
\lambda_{t,i}=\lambda_{i}, \tilde{d} < i \le d.
$$
Note in particular that the largest eigenvalue of $\Gamma_{t}$ has multiplicity $\tilde{d}$.\\
Moreover, we trivially have,
\be \label{bigger}
|Y|^2 \le |Y_{t}|^2
\ee 
Letting  $h$ and $h_{t}$ be the densities  of $|Y|^2$ and $|Y_{t}|^2,$ respectively,
we can infer from  Lemma 2(b) in \cite{Ein1991}  and  (\ref{(3.8)}) that
\be \label{237}
h(z) \ge h_{t}(z)\exp(-8d^3 z/t^2) , z > 0.
\ee
\begin{lem}\label{LemDensBound}
Assume that $t \ge 3d\lambda_{1}$ and that $Y_{t} $  is defined as above. Then the following holds:
\begin{itemize}
\item[(i)] If $0 < \delta \le t/4,$
we have:
 \be \label{eq39}
\PP\{|Y_{t}| \ge  t +\delta\} \le  4 C_3 e^{d/4} \exp(-t\delta/\lambda_{1}^2)  \PP\{|Y_{t}| \ge  t\}
\ee
\be \label{eq39a}
\PP\{|Y_{t}| \ge  t -\delta\} \le  6 C_3 \exp(t\delta/\lambda_{1}^2)  \PP\{|Y_{t}| \ge  t\}
\ee
\item[(ii)] If $t \ge C_5 \lambda_{1}$, where $C_5= 4\log(8C_3)^{1/2}\vee 3d$ and  $\beta  = \lambda_{1}^2  (\log (8C_3) +d/4),$ we have:
\be \label{lem2310}
\PP\{|Y_{t}| \ge t\} \le 2 \PP\{ t \le |Y_{t}| \le  t + \beta /t\}
\ee
\end{itemize}

	\end{lem}
\begin{proof}
(i) 
If $\tilde{d} < d$
 we  have by (\ref{eq44}) and Lemma \ref{LemDensBound-1},
\be \label{eq43a}
h_{t}(z) \le C_3(\tilde{d}) K(\Gamma_{t}^2) f_{\tilde{d}}(z/\lambda_{1}^2)/\lambda_{1}^2, z > 0. 
\ee
This is also correct if $d=\tilde{d}$ and $\Gamma_t$ is a diagonal matrix. Then we have $K(\Gamma_t^2)=1$ and the inequality is trivial since $C_3(d) > 1$ and $h_t(z)=f_{\tilde{d}}(z/\lambda_{1}^2)/\lambda_{1}^2, z > 0.$\\[.1cm]
If $\tilde{d} < d$ we have,
$$2\tilde{d}\E[|Y_{t}|^2]/(1-\lambda^2_{\tilde{d}+1}/\lambda^2_{1}) \le  2d^2\lambda_{1}^4 /(\lambda_{1}^2 - \lambda_{\tilde{d}+1}^2) \le  t^2/2$$
so that 
\be \label{eq44a}
h_{t}(z) \ge (K(\Gamma_{t}^2)/4)f_{\tilde{d}}(z/\lambda_{1}^2)/\lambda_{1}^2, z \ge t^2/2
\ee
which of course  also holds if $\tilde{d}=d.$\\
Further note that 
$$\PP\{|Y_{t}| \ge t + \delta\} \le \PP\{|Y_{t}|^2 \ge t^2 + 2t\delta\} =\int_{t^2}^{\infty} h_{t}(z+2t\delta)dz.$$
Combining inequalities (\ref{eq43a}) and  (\ref{eq44a}), we get for $z \ge t^2,$
\begin{eqnarray*}
h_{t}(z+2\delta t)/h_{t}(z) &\le& 4C_3(\tilde{d}) (1+2\delta/t)^{d/2} \exp(-t\delta/\lambda_{1}^2)\\
&\le& 4C_3(\tilde{d})\exp(d/4)\exp(-t\delta/\lambda_{1}^2).
\end{eqnarray*}
 $C_3(d)$ is monotone in $d$ by definition and we see that (\ref{eq39}) holds. \\[.1cm]
To prove (\ref{eq39a}) we note that 
$$\PP\{|Y_{t}| \ge t - \delta\} \le \PP\{|Y_{t}|^2 \ge t^2 - 2t\delta\} =\int_{t^2}^{\infty} h_{t}(z-2t\delta)dz.$$
 Since $\delta \le t/4,$ we get for $z \ge t^2,$
$$h_{t}(z-2\delta t)/h_{t}(z) \le 4C_3(\tilde{d}) (1-2\delta/t)^{-1/2} \exp(t\delta/\lambda_{1}^2)\le 6C_3(d) \exp(t\delta/\lambda_{1}^2)$$ 
and we see that (\ref{eq39a}) holds as well.\\[.2cm]
(ii) Applying  (\ref{eq39}) with $\delta =\beta/t= \lambda_{1}^2(\log(8C_3)+d/4)/t$, we see that
$$\PP\{|Y_{t}| \ge t\ + \beta/t\} \le  \PP\{|Y_{t}| \ge t\}/2,$$
where $\delta \le t/4$ since we are assuming that $t \ge C_5 \lambda_{1}.$
\end{proof}
\begin{lem}\label{lem234}
 Let $Y$ be an  $\mathcal{N}( 0, \Gamma^2)$-distributed random vector and let for any given  $t \ge C_5\lambda_{1}$ the random vector $Y_t$ be defined as in (\ref{Y_t}).
\begin{itemize}
\item [(i)] We have for any $0 \le \gamma < t^2/(4\lambda_1^2),$
$$\PP\{|Y| \ge t - \gamma \lambda_{1}^2/t\} \le 6C_3e^{\gamma} 
\PP\{|Y_{t}| \ge t\}.$$
\item[(ii)] Let $\beta$ be chosen as in Lemma \ref{LemDensBound}(ii). Then we also have:
$$\PP\{|Y_{t}| \ge t\} \le 2\exp(16d^3) 
\PP\{ t \le |Y| \le  t+\beta/t\}.$$
\end{itemize}
\end{lem}
\begin{proof}
(i) By relation (\ref{bigger}) we trivially have
 for any $y > 0,$
$$\PP\{|Y| \ge y\} \le  \PP\{|Y_{t}| \ge y\}.$$
Applying inequality (\ref{eq39a}) with  $\delta= \gamma\lambda_{1}^2/t,$  we can conclude that
$$
\PP\{|Y_t| \ge t - \gamma \lambda_{1}^2/t\} \le 6C_3e^{\gamma} \PP\{|Y_{t}| \ge t \}.
$$ 
(ii) Observe that we have by (\ref{lem2310})
$$
\PP\{|Y_{t}| \ge t \}  \le 2 \PP\{t \le |Y_t| \le  t+ \beta /t\}.$$
Recalling that $\beta/t \le t/4$  we  see that $( t +\beta /t)^2 \le 2t^2$ if $t \ge  C_5\lambda_{1}.$\\
Finally, relation (\ref{237}) implies that
\begin{eqnarray*}
&&\PP\{ t \le |Y_{t}| \le  t+ \beta /t\}\\
&\le & \int_{t^2}^{(t+\beta/t)^2} h_{t}(z) dz  \le  \exp(16d^3)\int_{t^2}^{(t+\beta/t)^2} h(z) dz
\\ & = &\exp(16d^3)\PP\{t \le |Y| \le t+ \beta /t\}. 
\end{eqnarray*}
and the lemma has been proven.
\end{proof}
\noindent Combining the two above lemmas, we obtain the main result of this section:
\begin{theo} \label{LemTailBoundsDistriGammanZn}
Let $Y: \Omega \to \R^d$ be an $\mathcal{N}(0,\Gamma^2)$-distributed random vector where $\Gamma$ is a symmetric non-negative definite matrix with eigenvalues $\lambda_1 \ge \ldots \ge \lambda_d$. \\Then we have for any $t \ge \tilde{C}_1 \lambda_{1}$
		
	\begin{itemize}
	\item[(i)]
	$\PP \{   \left| Y \right|  \ge  t\} 
			\le \tilde{C}_2 \prod_{i=2}^d \left\{ \frac{\lambda_{1}}{(\lambda_{1}^2-\lambda_{i}^2)^{1/2}} \wedge \frac{t}{\lambda_{1}}\right\}\frac{\lambda_{1}}{t}\exp\left(-\frac{t^2}{2\lambda_{1}^2}\right)
			 $
	\item[(ii)]
	$ \PP \{  |Y |  \ge  t\}
			  	\ge
		\tilde{C}_3 \prod_{i=2}^d \left\{ \frac{\lambda_{1}}{(\lambda_{1}^2-\lambda_{i}^2)^{1/2}} \wedge \frac{t}{\lambda_{1}}\right\}\frac{\lambda_{1}}{t}\exp\left(-\frac{t^2}{2\lambda_{1}^2}\right)$
	\item[(iii)] $\PP\{ | Y |  \ge   t - \gamma \lambda_{1}^2/t\}
	\le \tilde{C}_4 e^{\gamma} \PP\{    t \le | Y | 
		\le   t + \tilde{C}_5 \lambda_{1}^2/ t \}$ if $0 \le \gamma < t^2/(4\lambda_{1}^2)$
		\item[(iv)] $\PP\{    t \le | Y | 
		\le  \  t + \tilde{C}_5 \lambda_{1}^2/ t  \} \ge\tilde{C}_6  \prod_{i=2}^d \left\{ \frac{\lambda_{1}}{(\lambda_{1}^2-\lambda_{i}^2)^{1/2}} \wedge \frac{t}{\lambda_{1}}\right\}\frac{\lambda_{1}}{t}\exp\left(-\frac{t^2}{2\lambda_{1}^2}\right)$	\end{itemize}
where $\tilde{C}_i$, $1 \le i \le 6$ are positive constants  depending  on  $d$ only.
\end{theo}
\begin{proof} To verify (i) we first  note that
$$\PP\{|Y |  \ge  t\} \le \PP\{|Y_{t}| \ge t\}=\int_{t^2}^{\infty} h_{t}(z)dz.$$
If $\tilde{d} <d$, it follows from Lemma \ref{LemDensBound-1}, (\ref{eq44})  and (\ref{eq42}) that the last integral is
$$\le C_2 C_3  K(\Gamma_{t}^2)(t/\lambda_{1})^{\tilde{d} -2} \exp(-t^2/2\lambda_{1}^2).$$
If $\tilde{d}=d$,  the above bound follows directly from (\ref{eq42}). (Recall that $K(\Gamma^2)=1$ if $\Gamma$ is a diagonal matrix.)\\
By definition of $\tilde{d}$ we further have,
\be \label{Kgamma}
K(\Gamma^2_{t})(t/(2d\lambda_{1}))^{\tilde{d} -1} = \prod_{i=2}^d \{\lambda_{1}/(\lambda_{1}^2 -\lambda_{i}^2)^{-1/2} \wedge t/(2d\lambda_{1})\},
\ee
which is 
$$\le \prod_{i=2}^d \{\lambda_{1}/(\lambda_{1}^2 -\lambda_{i}^2)^{-1/2} \wedge t/\lambda_{1}\}.$$
(ii) Combining relations (\ref{eq44a}) and (\ref{eq42}), we can conclude that
$$\PP\{|Y_{t}| \ge t\} \ge \frac{C_1}{4} K(\Gamma^2_{t})(t/\lambda_{1})^{\tilde{d} -2} \exp(-t^2/2\lambda_{1}^2).
$$
where we  assume w.l.o.g. that $C_1=C_1(d)$ is non-increasing in $d$.\\
In view of (\ref{Kgamma}) we  have,
$$
K(\Gamma^2_{t})(t/(2d\lambda_{1}))^{\tilde{d} -1} \ge 2^{-d}d^{-d} \prod_{i=2}^d \{\lambda_{1}/(\lambda_{1}^2 -\lambda_{i}^2)^{-1/2} \wedge t/\lambda_{1}\}.
$$
 Combining this inequality with the fact that
 $$2\PP\{|Y| \ge t\} \ge \exp(-16d^3)\PP\{|Y_{t}| \ge t\},$$
 which follows from Lemma \ref{lem234}(ii), we readily obtain  (ii).\\
(iii) This part follows directly combining parts (i) and (ii) of Lemma \ref{lem234}.\\
(iv) Set $\gamma=0$ in (iii) and apply the lower bound given in (ii).
\end{proof}

\section{Remaining Proofs}\label{ProofMainTheorem}
By a standard argument (see, for instance,  Lemma~1 of \cite{Fe}) it is enough to prove our upper-lower class test for sequences $\phi_n$ which are non-decreasing and
satisfy 
\be \label{phi}
			\lambda_{n,1}^2 LLn  \le  \phi_n^2 \le 3\lambda_{n,1}^2 LLn.
\ee
Further recall that $\Gamma_n \to \Gamma$ implies that the eigenvalues $\lambda_{n,1} \ge \ldots \ge \lambda_{n,d}$ of $\Gamma_n$ converge to the corresponding eigenvalues $\lambda_{1} \ge \ldots \ge \lambda_{d}$ of $\Gamma$ (see (\ref{ga2_n})).\\[.2cm]	To simplify notation we set
$$\gamma_n:= \prod_{i=2}^d  \left(\frac{\lambda_{n,1}}{(\lambda_{n,1}^2-\lambda_{n,i}^2)^{1/2}} \wedge  \frac{\phi_n}{\lambda_{n,1}}\right) $$
If  $ d_1 < d$,  we have as $n \to \infty,$
$$\gamma_n \sim c  \prod_{i=2}^{d_1}  \left(\frac{\lambda_{n,1}}{(\lambda_{n,1}^2-\lambda_{n,i}^2)^{1/2}} \wedge  \frac{\phi_n}{\lambda_{n,1}}\right) =: c\gamma'_n,$$ 
where  $c:= \prod_{i=d_1+1}^d \frac{\lambda_{1}}{(\lambda_{1}^2-\lambda_{i}^2)^{1/2}} >0$. \\[.1cm] Thus, it is enough to prove  our upper-lower class test with $\gamma_n$ instead of $\gamma'_n.$\\[.2cm]
We need the following lemma which will be proven in in subsection \ref{SubsecAbout_Phi_n}.
\begin{lem}\label{LemSubseqPhi_nEnoughUpper_or_LowerCl}
	Let $\phi_n, n \ge 1$ be a non-decreasing sequence satisfying condition (\ref{phi}) and  given $\alpha >0,$ define
	$n_k =n_k(\alpha) =[\exp(\alpha_k/Lk] , k \ge 1.$ Assume that condition (\ref{main112}) holds.
	 Then we have
	$$
		\sum_{ k \ge 1 }  \frac{\gamma_{n_k}}{\phi_{n_k}}
		\exp\left( - \frac{ \phi_{n_k}^2 }{ 2\lambda_{n_k,1}^2 } \right) 
	< 
		\infty 
	\iff
		\sum_{n \ge 1} \frac{ \phi_n\gamma_n}{n} \exp\left( - \frac{ \phi^2_n }{ 2\lambda_{n,1}^2} \right)  
	< 
		\infty
	$$
\end{lem}

\subsection{Proof of Theorem \ref{main}: the upper class part}	\label{ProofMainTheoUppClass}
Let $\phi_n$ be a non-decreasing sequence satisfying condition (\ref{phi}) and 	$$
		\sum_{n=1}^{\infty} \frac{ \gamma_n \phi_n}{n} \exp\left( -\frac{ \phi_n^2 }{2\lambda_{n,1}^2} \right) 
	< 
		\infty.
	$$
Set $n_k =n_k(1)=[\exp(k/Lk)], k \ge 1.$ Lemma \ref{LemSubseqPhi_nEnoughUpper_or_LowerCl} then implies 
\be \label{5444}
\sum_{k=1}^{\infty} \gamma_{n_k}\frac{\lambda_{n_k,1}}{\phi_{n_k}}\exp\left( -\frac{ \phi_{n_k}^2 }{2\lambda_{n_k,1}^2} \right) < \infty.
\ee
Moreover,  it is easy to see that
\be \label{n_k-ratios}
\log(n_{k+1}/n_k) \sim 1/Lk \mbox{ as } k \to \infty
\ee
and that for large enough $k,$
\be \label{5333}
\frac{\lambda_{n_k,1}^2}{2} L k  \le \phi^2_{n_k} \le 4 \lambda_{n_k,1}^2 Lk \mbox{ and } 1/2 \le \lambda^2_{n_k,1}/\lambda^2_1\le 2
\ee
We have to show that
	$$
		\PP \{ | \Gamma_n T_n | >\phi_n \sqrt{n} \text{ infinitely often} \} 
	= 
		0.
	$$
By the Borel-Cantelli lemma it is   enough to prove that
\be
\sum_{k = 1 }^{\infty}   
	 \PP
	 	 \left(   \bigcup_{n=n_k +1 }^{n_{k+1}  } \{ |  \Gamma_n T_n |  >  \phi_n \sqrt{n} \} 
		 \right) < \infty.
\ee		 
To that end we will split this series into two parts:
 we first define a subsequence where the matrix sequence $\Gamma^2_n$ does not fluctuate too much for $n_k < n \le n_{k+1}$ and we show that the part of the above series determined by this subsequence is finite (see Step 1). In Step 2 we then show that the sum over the remaining $k$'s is finite as well. Since this will be a relatively sparse subsequence  coarse bounds for the probabilities involved will be sufficient at this stage.\\[.2cm]
\emph{Step 1} We start by defining the ``good'' subsequence as follows,
\be \label{Eq_ProofUppPartGoodSet} 
	 	\N_1 := 
		 \{ k : \Delta_k(1)  \le  \lambda_1^2(Lk)^{-3} \},
\ee
where $\Delta_k(1)$ is  as in (\ref{DELTA}).\\
Noting that
\begin{eqnarray*}
\bigcup_{n=n_k +1 }^{n_{k+1}  } \{ |  \Gamma_n T_n |  >  \phi_n \sqrt{n} \} 
&\subset&
 \bigcup_{n=n_k +1 }^{n_{k+1}  } \{ |  \Gamma_n T_n |  >  \phi_{n_k} \sqrt{n_k} \}\\
 &\subset& \bigcup_{n=n_k +1 }^{n_{k+1}  } \{ |  \Gamma_{n_k} T_n |  >  (\phi_{n_k} -\lambda_1^2/\phi_{n_k}) \sqrt{n_k} \}\\
 && \cup \bigcup_{n=n_k +1 }^{n_{k+1}  } \{ |  (\Gamma_n-\Gamma_{n_k}) T_n |  > \lambda_1^2 \sqrt{n_k}/\phi_{n_k} \}
		  \end{eqnarray*}
we see that
\be \label{5777}
 \PP\left(   \bigcup_{n=n_k +1 }^{n_{k+1}  } \{ |  \Gamma_n T_n |  >  \phi_n \sqrt{n} \} \right) 
 \le p_{k,1} + p_{k,2},
 \ee
 where
 $$p_{k,1}:= \PP\left\{\max_{1 \le n \le n_{k+1}} |  \Gamma_{n_k} T_n |  >  (\phi_{n_k} -\lambda_1^2/\phi_{n_k}) \sqrt{n_k}
\right\}$$
and
$$p_{k,2}:= \PP\left\{ \max_{n_k < n \le n_{k+1}}\|\Gamma_n - \Gamma_{n_k}\| |T_n| > \lambda_1^2\sqrt{n_k}/\phi_{n_k}\right\}.
$$
Set $\psi_n=\phi_n - \lambda_1^2/\phi_n, n \ge 1.$ Combining the L\'evy inequality  with Theorem \ref{LemTailBoundsDistriGammanZn}(i), we get that
$$p_{k,1} \le 2 \PP\{|\Gamma_{n_k}T_{n_{k+1}}| > \psi_{n_k}\sqrt{n_{k}}\}
\le 2\tilde{C}_2 \gamma_{n_k}\frac{\lambda_{n_k,1}}{\psi_{n_k}\sqrt{n_k/n_{k+1}}}\exp\left(-\frac{\psi_{n_k}^2}{2\lambda_{n_k,1}^2}\frac{n_k}{n_{k+1}}\right)
$$
Since $\psi_{n_k}\sim \phi_{n_k}$ and $n_{k+1}\sim n_k $ as $k \to \infty,$  and by definition of $\psi_n$ we obtain for large $k,$
$$p_{k,1} \le 3e^2\tilde{C}_2\gamma_{n_k}\frac{\lambda_{n_k,1}}{\phi_{n_k}}
\exp\left(-\frac{\phi_{n_k}^2}{2\lambda_{n_k,1}^2}\frac{n_k}{n_{k+1}}\right).$$
On account of relation (\ref{n_k-ratios}) we have  for large $k$, 
$$n_k/n_{k+1} \ge \exp(-2/Lk) \ge 1-2/Lk.$$
Recalling (\ref{5333}) we can  conclude that for large $k,$
$$
p_{k,1} \le  3e^6\tilde{C}_2\gamma_{n_k}\frac{\lambda_{n_k,1}}{\phi_{n_k}}
\exp\left(-\frac{\phi_{n_k}^2}{2\lambda_{n_k,1}^2}\right).$$
In view of (\ref{5444}) it is now clear that
\be \label{5888}
\sum_{k=1}^{\infty} p_{k,1} < \infty.
\ee
Combining the fact that $\|\Gamma_m -\Gamma_n\|^2 \le \|\Gamma_m^2 -\Gamma_n^2\|$ (see \cite{Bha},(X.2))  with (\ref{5333}) and the L\'evy inequality we get after a small calculation for large $k \in \N_1,$
$$p_{k,2} \le 2\PP\{|Z_1| \ge  Lk/3\},$$
 Using the upper  bound in  (\ref{eq42}), it  easily follows that
\be \label{5999}
\sum_{k \in \N_1} p_{k,2} < \infty.
\ee
\emph{Step 2} Suppose now that $k \notin \N_1$. 
Since $\Gamma_n \to \Gamma$, we have $\|\Gamma_n\|	 \le 2 \|\Gamma\|=2\lambda_1$ if $n \ge n_k$ and $k$ is large enough. It then follows that 
$$|\Gamma_n T_n| \le \|\Gamma_n\| |T_n| \le 2 \lambda_1 |T_n|.$$
We can conclude that for large $k,$
$$ \PP \left(   \bigcup_{n=n_k +1 }^{n_{k+1}  } \{ |  \Gamma_n T_n |  >  \phi_n \sqrt{n} \} 
		 \right) \le \PP\left\{\max_{1 \le n \le n_{k+1}} |T_n| \ge \phi_{n_k} \sqrt{n_k}/(2 \lambda_1)\right\}.$$
Using once more the L\'evy inequality, the fact that $n_k \sim n_{k+1}$ as $k \to \infty$  and (\ref{5333}) we find that if $k$ is large, the last probability is 
$$\le 2\PP\{|Z_1| \ge \phi_{n_k}/(3\lambda_1)\} \le 2\PP\{|Z_1| \ge (Lk)^{1/2} /6\}.$$ 
The last term is by (\ref{eq42}) of order $O((Lk)^{d/2 -1}k^{-1/72}).$\\[.1cm]
On the other hand we have by (\ref{DELTA}) and the definition of $\N_1$,
$$\sum_{k \not \in \N_1} (Lk)^{-3} k^{-\delta} < \infty, \forall \delta >0,$$
and we see that
$$\sum_{k \not \in \N_1}  \PP \left(   \bigcup_{n=n_k +1 }^{n_{k+1}  } \{ |  \Gamma_n T_n |  >  \phi_n \sqrt{n} \}  \right)  < \infty.$$
This implies in combination with (\ref{5777})-(\ref{5999}) the upper class part. 
\subsection{Proof of Theorem \ref{main}: the lower class part}\label{SubSecProofMainTheoLowClass}
		  Let $\phi_n$ be a non-decreasing sequence satisfying condition (\ref{phi}) and	$$
		\sum_{n=1}^{\infty} \frac{ \gamma_n \phi_n}{n} \exp\left( -\frac{ \phi_n^2 }{2\lambda_{n,1}^2} \right) =
		\infty.
	$$
Set $n_j=n_j(\alpha)=[\exp(\alpha j/Lj)], j\ge 1,$ where $\alpha > 0$ will be specified later. \\
From the definition of the subsequence $n_j$ it  follows that for $m \ge j \ge i_0$ (say)
\begin{eqnarray} \label{51515}
n_{m}/n_j &\ge& \exp(\alpha(m/\log m - j/\log j)/2) \nonumber \\
 &\ge& \exp(\alpha(m-j)/(4\log m )).
\end{eqnarray}
Consider the events
$$
A_n := \left\{   \sqrt{n} \phi_n < | \Gamma_n T_n | 
		< \sqrt{n} \left( \phi_n + \beta/ \phi_n  \right) \right\},
$$
where $\beta = 2\tilde{C}_5\lambda_1^2.$\\[.1cm]
 Since $\lambda_{n,1} \to \lambda_1$ we can infer from Theorem \ref{LemTailBoundsDistriGammanZn}(iv)  and Lemma~\ref{LemSubseqPhi_nEnoughUpper_or_LowerCl} that 
\be \label{eq310zzz}
\sum_{ j \ge 1} \PP ( A_{n_j } ) = \infty
\ee 
 Notice that 
$$
\PP \left\{ A_{n_j } \text{ infinitely often} \right\} 
	\le \PP  \left\{  | \Gamma_{n_j } T_{n_j } | 
			  >  \phi_{n_j } \sqrt{ n_j } \text{ infinitely often} 
			   \right\}.
$$
An application of the Hewitt-Savage 0-1 law shows that the latter probability can only be 0 or 1. Consequently, it suffices to show   that there exists a constant $K_0 >0$ such that for all (large) $i \ge 1$ 
\be \label{limsup}
\PP ( \cup_{j=i}^{\infty} A_{n_j } )  \ge K_0.
\ee
\noindent To that end we need a
	good lower bound for $\PP ( \cup_{j=i}^k A_{n_j })$ for any given natural numbers $i < k.$ Similarly as in the classical papers \cite{Er} and \cite{Fe0}  this 
can be accomplished by providing good upper bounds
	for $$\PP ( A_{n_j } \cap A_{n_m } ), i\le j < m \le k.$$ 
In our present setting this will be only possible for a subsequence of $\{n_j\}$. \\So we have to add an extra Step 1 where we have  to determine a suitable subsequence. In Step 2 we provide upper bounds  for $\PP ( A_{n_j } \cap A_{n_m } )$ for numbers $n_j$ and $n_m$ in this subsequence. In Step 3 we finally prove (\ref{limsup}) where we take the union only over the $n_j$'s in  our subsequence.\\[.2cm]
\emph{Step 1} We first note that for $j < m,$
$$A_{n_j } \cap A_{n_m } \subset  A_{n_j }  \cap (B_{j,m} \cup B'_{j,m}),$$
where
$$
B_{j,m} =	\bigg\{ | \Gamma_{n_m } (T_{n_m } -T_{n_j } ) |  > \sqrt{n_m}  (\phi_{n_m} - \beta/\phi_{n_m})  
		-\sqrt{n_j} 
		\left(    \phi_{n_j} + \beta / \phi_{n_j}  \right)\bigg\} 
$$
and
$$B'_{j,m} =\left\{| (\Gamma_{n_m} -\Gamma_{n_j } ) T_{n_j } |
		 >\beta \sqrt{n_m}  /\phi_{n_m } \right\}.
$$
We further set 
$$C_n:=\{|T_n| > 2\sqrt{nLLn}\}, n \ge 1$$
Using, for instance, (\ref{eq42}) we see that
\be \label{eq31xx}
\sum_{n=1}^{\infty} \PP(C_n) < \infty
\ee
\begin{lem} \label{lem32xx}
 There exists a $j_0 \ge 1$ such that we have for any $j \ge j_0,$ $$B'_{j,m} \subset C_{n_j}$$
 provided that one of the two following conditions is satisfied
 \begin{itemize}
\item [(i)]  $m \ge j+(\log j)^3$ 
\item[(ii)] $\|\Gamma_{n_m} - \Gamma_{n_j}\| \le (\log j)^{-2}$ and $j < m \le j+(\log j)^3.$
\end{itemize}
\end{lem}
\noindent {\bf Proof} 
(i) In this case we have if $j$ is sufficiently large, 
\be \label{eqzx}
\sqrt{n_m/n_j} \ge 8 \lambda_1^3(\log \log n_m)^{3/2} \ge \phi^3_{n_m}
\ee
To see this, we set $m_j := [j + (\log j)^3] +1.$\\
By monotonicity of the function $x \mapsto \sqrt{x}/(\log \log x)^{3/2}, x \ge e^e$, it is  enough to check that
$$\sqrt{n_{m_j} /n_j} \ge 8 \lambda_1^3(\log \log n_{m_j})^{3/2}.$$
This follows easily from (\ref{51515}) and the fact that $\log m_j /\log j \to 1$ as $j \to \infty.$\\[.1cm]
Assuming that $\|\Gamma_{n_m} - \Gamma_{n_j}\| \le 1, j, m \ge j_0$ (recall that  $\Gamma_n \to \Gamma$) we have 
 \be \label{Eq_LowPartUppBoundForIntersectionEq1}
	 	| (\Gamma_{n_m} -\Gamma_{n_j} ) T_{n_j }(\omega) | 
	 \le 
	 	2\sqrt{n_j  LLn_j}, \omega \not \in C_{n_j } 		\ee
		which is for large enough $j$,
$$
	\le 
	(4\lambda^3_1)^{-1}\sqrt{n_m} / LLn_m     \le \beta \sqrt{n_m}  /\phi_{n_m}
	$$		
(ii) 	In this case we get for $\omega \not \in C_{n_j}$ and large $j,$
$$ | (\Gamma_{n_m} -\Gamma_{n_j} ) T_{n_j }(\omega) | 
	 \le 
	 	2 \sqrt{n_j}( \log j)^{-3/2}\le \beta\sqrt{n_m}/\phi_{n_m}$$ and the lemma has been proven. \qed\\[.2cm]
Consider the subsequence 
$$\mathbb{N}_2: =\left\{j \ge 1: \max_{j+1 \le m \le j+(\log j)^3}\|\Gamma_{n_m}- \Gamma_{n_{j}}\| < (\log j)^{-2}\right\}.$$
From Lemma \ref{lem32xx} it is then clear that  $B'_{j,m} \subset C_{n_j}$ for all $m > j $ if $j \in \mathbb{N}_2$ is large enough. \\We next show that the sets $A_{n_j}, j \not \in \mathbb{N}_2$ will be irrelevant for our purposes. \\[.1cm]To that end we
		need the following lemma.
\begin{lem} \label{lem251} Assume that condition (\ref{main112}) is satisfied. Let $\ell_j$ be a non-decreasing sequence of natural numbers satisfying $\ell_j \le j/2$ and
 $\ell_j/\ell_{[j/2]} = O(1)$ as $j \to \infty.$ Then we have for any $\delta>0,$
 $$\sum_{j=1}^{\infty} \ell_j^{-1} j^{-\delta}\max_{j < m \le j+\ell_j} \|\Gamma^2_{n_{m}}- \Gamma^2_{n_{j}}\|< \infty.$$ 
 \end{lem}
 \begin{proof}
 Setting $\Delta'_k:=\|\Gamma^2_{n_{k+1}}-\Gamma^2_{n_{k}}\|, k \ge 1,$ we see
 that the above series is 
  $$\le \sum_{j=1}^{\infty} \ell_j^{-1} j^{-\delta} \sum_{k=j}^{j + \ell_j -1}\Delta'_k
 \le \sum_{k=1}^{\infty} \Delta_k(\alpha) \sum_{j=k-l_k+1}^k  \ell_j^{-1} j^{-\delta}.$$
 As we have $\ell_k \le k/2$, we can conclude that
 $$
 \sum_{j=k-\ell_k+1}^k  \ell_j^{-1} j^{-\delta} \le \ell_k \ell_{[k/2]}^{-1} (k/2)^{-\delta}=O(k^{-\delta})
 $$
 and the assertion of the lemma follows from assumption (\ref{main112}).
\end{proof}
\noindent Using the inequality $|\Gamma_{n_j} T_{n_j}| \le \|\Gamma_{n_j}\|\,|T_{n_j}|=\lambda_{n_j,1}\,|T_{n_j}|$  and (\ref{phi}), we have
$$\PP(A_{n_j}) \le \PP\{ |\Gamma_{n_j} T_{n_j}| > \sqrt{n_j}\phi_{n_j}\} \le \PP\{|Z_1| \ge \sqrt{ LLn_j}\}$$
Employing the upper bound in (\ref{eq42}) we see by definition of the subsequence $n_j$ that
\be \label{55514}
\PP(A_{n_j}) = O((\log j)^{(d-1)/2 } j^{-1/2}) \mbox{ as } j \to \infty.
\ee
By relation (X.2) in \cite{Bha} we have  for $j \not \in \mathbb{N}_2$
$$(\log j)^{-2}\le \max_{j < m \le j + (\log j)^3} \|\Gamma_{n_m} - \Gamma_{n_j}\| \le  \max_{j < m \le j + (\log j)^3} \|\Gamma^2_{n_m} - \Gamma^2_{n_j}\|^{1/2}$$
which in combination with Lemma \ref{lem251} implies
$$\sum_{j \not \in \mathbb{N}_2} (\log j)^{-7} j^{-\delta} < \infty, \forall \delta >0.$$
Combining this with (\ref{55514}) we see that $\sum_{j \not \in \mathbb{N}_2}\PP(A_{n_j})< \infty.$\\
Recalling (\ref{eq310zzz}) we find that
\be \label{515ab}
\sum_{j  \in \mathbb{N}_2}  \PP(A_{n_j} )  = \infty.
\ee
We will show that this implies for some positive constant $K_0$ and $i \ge 1,$
\be \label{516ab}
\PP\left(\bigcup_{j \in [i,\infty[ \cap \mathbb{N}_2} A_{n_j }\right) \ge K_0
\ee
which trivially implies (\ref{limsup}). \\[.2cm]
\emph{Step 2} Take two natural numbers $i < k$. Then we have,
\begin{align*} 
		\PP \left( \bigcup_{j \in [i,k] \cap \mathbb{N}_2}  A_{n_j }  \right) 
	&= 
		\sum_{j \in [i,k] \cap \mathbb{N}_2}  \PP \left( A_{n_j}  \backslash \bigcup_{m \in ]j,k] \cap \mathbb{N}_2}A_{n_m } \right)  \\
	&=
		\sum_{j \in [i,k] \cap \mathbb{N}_2} \left( \PP ( A_{n_j }  ) 
					-  
					\PP \left( \bigcup_{m \in ]j,k] \cap \mathbb{N}_2} ( A_{n_j } \cap A_{n_m } ) \right) 
				 \right)
\end{align*} 
If $i \in \mathbb{N}_2$ is large enough, we can infer from Lemma \ref{lem32xx} that
$$\bigcup_{m \in ]j,k] \cap \mathbb{N}_2} ( A_{n_j } \cap A_{n_m } ) \subset C_{n_j} \cup
\bigcup_{m \in ]j,k] \cap \mathbb{N}_2} ( A_{n_j } \cap B_{j,m} ) $$
It then follows by independence of the events $A_{n_j }$ and $B_{j,m}$ that
\be \label{319aa}
 \PP\left(\bigcup_{j \in [i,k] \cap \mathbb{N}_2} A_{n_j } \right)
\ge  \sum_{j \in [i,k] \cap \mathbb{N}_2}\PP(A_{n_j }) \bigg(1- \sum_{m \in ]j,k] \cap \mathbb{N}_2} \PP(B_{j,m})\bigg) -  \epsilon_i,\ee
where $\epsilon_i := \sum_{j=i}^{\infty} \PP(C_{n_j}) \to 0$ as $i \to \infty$ (see (\ref{eq31xx})). \\ [.3cm]  We now derive upper bounds for  $\PP(B_{j,m})$ if  $m > j$ and $j \in \mathbb{N}_2.$\\[.2cm]
{\bf Case 1: }    $j < m \le j + (\log j)^3$\\
(a) \emph{We first look at the ``small'' $m$'s  satisfying}
 $ n_m /n_j \le 4$ 	\\[.1cm]
By monotonicity of  the sequence $\phi_n$ we then have $\sqrt{n_m}/\phi_{n_m} \le 2\sqrt{n_j}/\phi_{n_j}$ and
\begin{eqnarray} \label{51717a}
&&\sqrt{n_m}  (\phi_{n_m} - \beta/\phi_{n_m})  -\sqrt{n_j} ( \phi_{n_j} + \beta / \phi_{n_j})\nonumber\\
&\ge&(\sqrt{n_m}-\sqrt{n_j})\phi_{n_m} - 3\beta \sqrt{n_j}/\phi_{n_j}.
\end{eqnarray}
Applying (\ref{51515}) with $m=j+1$, we find that
$$\sqrt{n_{j+1}/n_j} -1 \ge \alpha /(8 \log (j+1) )\ge  6\beta/\phi^2_{n_j},$$
if we choose $\alpha \ge 384 \tilde{C}_5 (= 192 \beta /\lambda_1^2)$.\\ (Here we have assumed that $\phi_{n_j} \ge (\lambda_1/2)\sqrt{\log (j+1)}$ if $j \ge i_0.$ This is possible by condition (\ref{phi}) and since $\lambda_{n,1} \to \lambda_1.$)\\[.1cm]
Returning to (\ref{51717a}) we get for $i_0 \le j < m,$
$$
\sqrt{n_m}  (\phi_{n_m} - \beta/\phi_{n_m})  -\sqrt{n_j} ( \phi_{n_j} + \beta / \phi_{n_j})
\ge\frac{1}{2}(\sqrt{n_m}-\sqrt{n_j})\phi_{n_m}
$$
Using the inequality $|\Gamma_{n_m} (T_{n_m}-T_{n_j})| \le \lambda_{n_m,1} |T_{n_m}-T_{n_j}|$ in conjunction with
 (\ref{phi}) and the fact that $\log \log n_j \sim \log j$ as $j \to \infty$ we find that for $m > j \ge i_0$ (which can be enlarged if necessary)
$$\PP(B_{j,m})\le \PP\left\{|T_{n_m}-T_{n_j}| > \frac{1}{3}(\sqrt{n_m}-\sqrt{n_j})\sqrt{\log m}\right\}.$$
As we have $T_{n_m}-T_{n_j}\stackrel{d}{=}\sqrt{n_m - n_j}\,Z_1$ we see that the last probability is
$$\le \PP\left\{|Z_1| \ge \frac{1}{3}\frac{\sqrt{ \log m}\sqrt{n_m - n_j}}{\sqrt{n_m} + \sqrt{n_j}}\right\} \le 
 \PP\left \{|Z_1| \ge \frac{1}{9}\sqrt{ \log m}\sqrt{n_m/n_j -1}\right\}.$$
 Recall that $n_m/n_j \le 4.$ \\In view of  inequality (\ref{51515}) it is now clear  that we have if $m > j \ge i_0,$
 $$\PP(B_{j,m}) \le \PP\left\{|Z_1| \ge \frac{1}{18}\sqrt{\alpha} \sqrt{m-j}\right\}.$$
Using the following simple inequality which follows from Lemma 4 in \cite{Ein1991}
\be \label{simple}
\PP\{|Z_1| \ge t\} \le \exp(-t^2/8), t \ge 2\sqrt{d} 
\ee 
and setting $\alpha = 54^2 d \vee 384 \tilde{C}_5$, we find that
\be \label{51616}
 \PP(B_{j,m}) \le  3^{j-m}, m > j \ge i_0,
 \ee
if  $n_m/n_j \le 4$ and $j < m \le m+(\log j)^3.$\\[.3cm]
 (b) \emph{We now look at the ``larger'' $m$'s  where} $n_m/n_j >4.$\\
Since the function $x\mapsto x + \beta/x$ is increasing for $x \ge \sqrt{\beta}$, we obtain in this case
 $$\sqrt{n_m}  (\phi_{n_m} - \beta/\phi_{n_m})  -\sqrt{n_j} ( \phi_{n_j} + \beta / \phi_{n_j})
 \ge \frac{1}{2}\sqrt{n_m}(\phi_{n_m} - 3\beta/\phi_{n_m}) \ge \frac{1}{3}\sqrt{n_m}\phi_{n_m}
$$
for  $m > j \ge i_0$ and large enough $i_0$.\\
Recall  that  $|\Gamma_{n_m}(T_{n_m}-T_{n_j})| \le \lambda_{n_m,1}|T_{n_m}-T_{n_j}|$ and relation (\ref{phi}). \\We then get if  $i_0$ is sufficiently large (and $n_m/n_j > 4$), 
$$\PP(B_{m,j}) \le \PP\{|Z_1| > \sqrt{\log\log n_m}/3\} \le (\log n_m)^{-1/72}, m > j \ge i_0$$
where the last inequality follows from (\ref{simple}). \\
Combining this bound with (\ref{51616}), we obtain for $j \ge i_0$ if $i_0$ is sufficiently large, 
\be \label{51717}
\sum_{m=j+1}^{j + (\log j)^3} \PP(B_{j,m}) \le \sum_{k=1}^{\infty} 3^{-k} + (\log j)^3  (\log n_j)^{-1/72} \le 2/3
\ee
since $(\log j)^3  (\log n_j)^{-1/72}\to 0$ as $j \to \infty$ by definition of the subsequence $\{n_j\}$.\\[.2cm]
{\bf Case 2}  $m > j+ (\log j)^3.$\\
We now have if $j$ is sufficiently large, 
$\sqrt{n_m/n_j} \ge \phi^3_{n_m}  $ (see (\ref{eqzx})).\\
Since $\phi_{n_j} + \beta/\phi_{n_j} \le 2 \phi_{n_j} \le 2 \phi_{n_m}$ if $j$ is large enough,
we get if $j \ge j_0$ (say)
\begin{eqnarray*}
\PP(B_{j,m})& \le& \PP\{ | \Gamma_{n_m } (T_{n_m } -T_{n_j } ) |  > \sqrt{n_m}  (\phi_{n_m} - 3\beta/\phi_{n_m}) \}\\
&\le& \PP\{ | \Gamma_{n_m } T_{n_m } |  > \sqrt{n_m}  (\phi_{n_m} - \gamma \lambda_{n_m,1}^2/\phi_{n_m}) \},
 \end{eqnarray*}
where we set $\gamma = 8\tilde{C}_5(=4\beta \lambda_1^{-2}).$	Recall that $\lambda_{n,1} \to \lambda_1$ as $n \to \infty.$	\\[.1cm]
Applying  Theorem \ref{LemTailBoundsDistriGammanZn}(iii) we get  for large enough $j,$
\be \label{51818}
\PP(B_{j,m}) \le K_1\PP(A_{n_m}) \mbox{ if } m > j+(\log j)^3,
\ee
where $K_1:=\tilde{C}_4 \exp(8\tilde{C}_5).$\\
\emph{Step 3}	  Combining  (\ref{51717}) and (\ref{51818}), we infer from (\ref{319aa}) that for $k > i \ge i_0,$
$$\PP\left(\bigcup_{j \in [i,k] \cap \mathbb{N}_2} A_{n_j } \right)
\ge  \sum_{j \in [i,k] \cap \mathbb{N}_2}\PP(A_{n_j }) \bigg(1/3 - K_1\sum_{m \in [i, k] \cap \mathbb{N}_2} \PP(A_{n_m})\bigg) -\epsilon_i.$$	  
Assume that we have chosen $i_0$ large enough so that $\PP(A_{n_j}) \le 	  1/(8 K_1), j \ge i_0.$\\ In view of (\ref{515ab}), we can find for any $i \ge i_0$ a  $k \in \mathbb{N}_2$
 such that $k \ge i$ and
$$\sum_{j \in [i,k] \cap \mathbb{N}_2}\PP(A_{n_j }) > \frac{1}{8K_1}.$$
Denote the smallest $k$ with this property by $k_i$. 
Then it is easy to see that
$$\sum_{j \in [i,k_i] \cap \mathbb{N}_2}\PP(A_{n_j }) \le \PP(A_{n_{k_i}}) +\sum_{j \in [i,k_i[ \cap \mathbb{N}_2}\PP(A_{n_j }) \le  \frac{1}{4K_1}.$$ 
We obtain for large enough $i \in \mathbb{N}_2,$
$$\PP\left(\bigcup_{j \in  \mathbb{N}_2 \cap [i,\infty[ } A_{n_j } \right) \ge\PP\left(\bigcup_{j \in [i,k_i] \cap \mathbb{N}_2} A_{n_j } \right)
\ge \frac{1}{12} \sum_{j \in [i, k_i] \cap \mathbb{N}_2} \PP(A_{n_j})-\epsilon_i \ge \frac{1}{100K_1}.$$
This clearly implies (\ref{516ab}) with $K_0= 1/(100K_1).$
\subsection{Proof of Corollary \ref{CFeller} and (\ref{eq24}) }
As in the classical case it is enough to prove the corollary for sequences $\phi_n$ satisfying condition (\ref{phi}).\\ Setting $\psi_{n,\beta} = \phi_n + \beta/\phi_n$ for $\beta \ne 0$, one easily checks that 
\begin{eqnarray*}
&&\sum_{n=1}^{\infty} \frac{\phi_n}{n \lambda_{n,1}}\prod_{i=2}^{d _1} \left(\frac{\lambda_{n,1}}{(\lambda_{n,1}^2-\lambda_{n,i}^2)^{1/2}} \wedge  \frac{\phi_n}{\lambda_{n,1}}\right) \exp\left(-\frac{\phi_n^2}{2\lambda_{n,1}^2} \right) < \infty\\
&\Longleftrightarrow&\sum_{n=1}^{\infty} \frac{\psi_{n,\beta}}{n \lambda_{n,1}}\prod_{i=2}^{d _1} \left(\frac{\lambda_{n,1}}{(\lambda_{n,1}^2-\lambda_{n,i}^2)^{1/2}} \wedge  \frac{\psi_{n, \beta}}{\lambda_{n,1}}\right) \exp\left(-\frac{\psi_{n,\beta}^2}{2\lambda_{n,1}^2} \right) < \infty
\end{eqnarray*}
By a standard argument (see, for instance, Sect. 5 in \cite{E-5}), we then can infer from Theorem C  that
$$\PP\{|S_n| \le \lambda_{n,1} \sqrt{n}\phi_n \mbox{ eventually}\}= \PP\{|\Gamma_n T_n| \le \lambda_{n,1}\sqrt{n} \phi_n \mbox{ eventually}\}$$
for any non-decreasing sequence $\phi_n$ where $\Gamma_n$ is defined as in (\ref{gamma_n}). Note that $\lambda_{n,1}$ is non-decreasing since $\Gamma_n^2 - \Gamma_m^2$ is non-negative definite if $m \le n.$ Consequently, we have $\lambda^2_{m,i} \le \lambda^2_{n,i}, 1 \le i \le d$ (see, for instance, Lemma A.1 in \cite{Ein1991}.) \\[.2cm]
In view of Theorem \ref{main} we only need to check condition (\ref{main112}). If $m < n$ we have 
\begin{eqnarray*}
\|\Gamma_n^2 - \Gamma_m^2\|&=&\sup_{|v| =1} |\langle v, (\Gamma_n^2 - \Gamma_m^2) v\rangle|\\
&=&\sup_{|v| =1}\E\langle X_1, v \rangle^2I_{\{c_m < |X_1| \le c_n\}} \le \E |X_1|^2  I_{\{c_m < |X_1| \le c_n\}}
\end{eqnarray*}
and we see that for any $\alpha >0,$
$$\sum_{k=1}^{\infty} \Delta_k(\alpha) \le \E|X_1|^2 < \infty$$
which is more than we need.\\[.3cm]
{\bf Proof of  (\ref{eq24})} This is trivial if $d_1=1$. If $d_1 \ge 2,$  we apply Theorem \ref{main} with the sequence $\lambda_{n,d}\phi_n$ (which is monotone since $\lambda_{n,d}$ is non-decreasing). Noting that $\lambda_{n,1} \to \lambda_1 > 0$, we find that
$$\PP\{|S_n| \ge \lambda_{n,d_1}\phi_n \mbox{ infinitely often}\}=1$$
provided that
\be \label{eq331}
\sum_{n=1}^{\infty} \frac{\phi_n}{n} \prod_{i=2}^{d_1} \left((1-\rho_{n,i}^2)^{-1/2} \wedge \phi_n \right)\exp((1-\rho_{n,d_1}^2)\phi_n^2/2) \exp(-\phi_n^2/2) =\infty,
\ee 
where $\rho_{n,i}^2 = \lambda_{n,i}^2/\lambda_{n,1}^2, 2 \le i \le d_1.$\\
Further note that
\begin{eqnarray*}
&&\prod_{i=2}^{d_1} \left((1-\rho_{n,i}^2)^{-1/2} \wedge \phi_n \right)\exp((1-\rho_{n,d_1}^2)\phi_n^2/2)\\
&\ge&  \left((1-\rho_{n,d_1}^2)^{-(d_1-1)/2}\exp((1-\rho_{n,d_1}^2)\phi_n^2/2)\right) \wedge \left(\phi_n^{d_1 -1} \exp((1-\rho_{n,d_1}^2)\phi_n^2/2)\right)\\
&\ge& \phi_n^{d_1-1} /\sqrt{(d_1-1)!}
\end{eqnarray*}
where we have used the trivial inequality $e^{x/2} \ge x^{k/2}/\sqrt{k!}$ for $k=d_1-1.$\\
It is now evident that if the series 
$$\sum_{n=1}^{\infty} \frac{\phi_n^{d_1}}{n} \exp(-\phi_n^2/2)$$
is divergent,  (\ref{eq331}) holds and we have proven (\ref{eq24}).
\subsection{Proof of Lemma \ref{LemSubseqPhi_nEnoughUpper_or_LowerCl}}\label{SubsecAbout_Phi_n}
We need to show
$$\sum_{ k \ge 1 }  \frac{\gamma_{n_k}}{\phi_{n_k}}
		\exp\left( - \frac{ \phi_{n_k}^2 }{ 2\lambda_{n_k,1}^2 } \right) 
	< 
		\infty 
	\iff \sum_{k \ge 1} a_k <\infty,$$
	where
$$a_k := \sum_{n_k < n \le n_{k+1}}\frac{ \phi_n\gamma_n}{n} \exp\left( - \frac{ \phi^2_n }{ 2\lambda_{n,1}^2} \right), k \ge 1 .$$
Recall definition (\ref{DELTA}) and consider the subsequence
$$\mathbb{N}_3 :=\{ k \ge 1:\Delta_k(\alpha)\le \lambda_1^2/Lk\}.$$
From condition (\ref{main112}) it is then clear that
\be \label{xx22}
\sum_{k \not \in \mathbb{N}_3} k^{-\delta} < \infty, \;\forall \delta >0.
\ee
Noting that $\gamma_n \le (\phi_n/\lambda_{n,1})^{d-1}$, it easily follows from (\ref{phi}) that
$$ \gamma_{n_{k+1}}
		\exp\left( - \frac{ \phi_{n_{k+1}}^2 }{ 2\lambda_{n_{k+1},1}^2 }\right) \le 3^{(d-1)/2} (LLn_{k+1})^{(d-1)/2} (Ln_{k+1})^{-1/2}
$$
so that we have by (\ref{xx22})
\be \label{a_k0}
\sum_{ k \not \in \mathbb{N}_3 }   \frac{\gamma_{n_{k+1}}}{ \phi_{n_{k+1}}}
		\exp\left( - \frac{ \phi_{n_{k+1}}^2 }{ 2\lambda_{n_{k+1},1}^2 } \right) 
	< \infty 
\ee
Using a similar argument as above and that $\log(n_{k+1}/n_k ) \sim \alpha(Lk)^{-1}$ we find that
\be \label{a_k1}
\sum_{ k \not \in \mathbb{N}_3 } a_k < \infty.
\ee
So it suffices to show that we have for some positive constants $C'_1 \le C'_2$ and large enough $k \in \mathbb{N}_3,$
\be \label{C_1,2}
C'_1\, \frac{\gamma_{n_{k+1}}}{ \phi_{n_{k+1}}}
		\exp\left( - \frac{ \phi_{n_{k+1}}^2 }{ 2\lambda_{n_{k+1},1}^2 } \right) 
\le a_k \le C'_2\,\frac{\gamma_{n_k}}{\phi_{n_k} }
		\exp\left( - \frac{ \phi_{n_k}^2 }{ 2\lambda_{n_k,1}^2 } \right).\ee 
We  first look at the upper bound. 
Since   any function of the form $$x \mapsto\prod_{i=2}^d  (c_i \wedge x/b)\exp(-x^2/2b^2)$$  (where $c_i> 0$) is decreasing for $x \ge b\sqrt{d}$ we have for $n_k \le n \le n_{k+1}$ and large enough $k,$
$$\gamma_n \exp(-\phi_n^2/(2\lambda_{n,1}^2)) \le \gamma'_{n,k} \exp(-\phi_{n_k}^2/(2\lambda_{n,1}^2)),$$
where 
$$ \gamma'_{n,k}: = \prod_{i=2}^d  \left(\frac{\lambda_{n,1}}{(\lambda_{n,1}^2-\lambda_{n,i}^2)^{1/2}} \wedge  \frac{\phi_{n_k}}{\lambda_{n,1}}\right) .$$
As $\lambda_{n,1}^2 \to \lambda_1^2 >0$, we have for $n \ge n_k$ and large $k,$
$\lambda_{n,1}^{d-1} \le 2 \lambda_{n_k,1}^{d-1}\le 4 \lambda_{n,1}^{d-1}$ so that 
$$\gamma'_{n,k} \le \gamma''_{n,k}: = 2\prod_{i=2}^d  \left(\frac{\lambda_{n_k,1}}{(\lambda_{n,1}^2-\lambda_{n,i}^2)^{1/2}} \wedge  \frac{\phi_{n_k}}{\lambda_{n_k,1}}\right).$$
Moreover, we have for $n_k \le m \le n \le n_{k+1},$
$$|\log(\gamma''_{n,k}/\gamma''_{m,k})|\le \sum_{i=2}^d |\log(\{\lambda_{n,1}^2-\lambda_{n,i}^2)\}\vee \tilde{\phi}_{n_k}^{-2}) -
\log(\{\lambda_{m,1}^2-\lambda_{m,i}^2\}\vee \tilde{\phi}_{n_k}^{-2}))|/2,$$
where $\tilde{\phi}_{n_k}=\phi_{n_k}/\lambda_{n_k,1}^2 \le \sqrt{3LLn_k}/\lambda_{n_k,1}\le 4\sqrt{Lk}/\lambda_{1}$ if $k$ is large enough.\\
From Lemma A.1.i in \cite{Ein1991} we can infer that 
$$|\lambda_{m,i}^2 - \lambda_{n,i}^2| \le \|\Gamma^2_m - \Gamma^2_n\|, 1 \le i \le d.$$
 An application of the mean value theorem then shows that the above sum is $$\le \tilde{\phi}_{n_k}^{2}\sum_{i=2}^d|\lambda_{m,1}^2-\lambda_{m,i}^2 -(\lambda_{n,1}^2-\lambda_{n,i}^2)| /2\le 16(d-1)Lk \|\Gamma_m^2 - \Gamma_n^2\|/\lambda_1^2 . $$
Choosing $m=n_k$ and noting that $\gamma''_{n_k,k}=2\gamma_{n_k}$, we obtain for large $k \in \mathbb{N}_3,$
$$\gamma'_{n,k}/\gamma_{n_k} \le 2\exp(16(d-1)), n_k \le n \le n_{k+1}.$$
Further note that for large enough $k,$
\begin{eqnarray*}
&&\exp(-\phi_{n_k}^2/2 \lambda_{n,1}^2)/\exp(-\phi_{n_k}^2/2 \lambda_{n_k,1}^2) \\
&=&\exp(\phi_{n_k}^2 (\lambda_{n,1}^2 - \lambda_{n_k,1}^2)/(2\lambda^2_{n_{k,1} }\lambda^2_{n,1}))\\&\le& \exp(3Lk \|\Gamma_n^2 - \Gamma_{n_k}^2\|/\lambda_1^2)
\end{eqnarray*}

We see that if $k \in \mathbb{N}_3$ is large enough, we have
\begin{eqnarray*}
a_k &\le& 2\exp(16d-13)\phi_{n_{k+1}}\gamma_{n_k}\exp(-\phi^2_{n_k}/(2\lambda^2_{n_k,1}))\sum_{n=n_k +1}^{n_{k+1}} n^{-1}\\
&\le&2\exp(16d-13)\phi_{n_{k+1}}\log(n_{k+1}/n_k)\gamma_{n_k}\exp(-\phi^2_{n_k}/(2\lambda^2_{n_k,1}))
\end{eqnarray*}
From condition (\ref{phi}) and the fact that $\lambda_{n,1} \to \lambda_1$ it easily follows that
$\phi_{n_{k+1}}/\phi_{n_k}$ remains bounded. Since $\log(n_{k+1}/n_k ) \sim \alpha(Lk)^{-1}$ and $\phi^2_{n_k} \le  4 \lambda_1^2 Lk $  for large $k$,  the upper bound in (\ref{C_1,2}) follows.\\[.2cm]
The lower bound follows by an  obvious modification of the above arguments and   Lemma \ref{LemSubseqPhi_nEnoughUpper_or_LowerCl} has been proven.

 \bibliographystyle{amsplain}

\end{document}